\documentclass[12pt,english,reqno]{article}
\usepackage{lmodern}
\usepackage[T1]{fontenc}
\usepackage[latin9]{inputenc}
\usepackage[a4paper]{geometry}
\usepackage{amsthm}
\usepackage{amsmath}
\usepackage{amssymb}
\usepackage{setspace}
\usepackage[authoryear]{natbib}
\usepackage{enumerate}
\usepackage{graphicx}

\makeatletter
\numberwithin{equation}{section}
\numberwithin{figure}{section}
\theoremstyle{plain}
\newtheorem{thm}{\protect\theoremname}[section]
  \theoremstyle{remark}
  \newtheorem{rem}[thm]{\protect\remarkname}
    \theoremstyle{lemma}
  
  \theoremstyle{definition}
  \newtheorem{example}[thm]{\protect\examplename}
    \theoremstyle{definition}
  \newtheorem{dfn}[thm]{Definition}
      \theoremstyle{plain}
  \newtheorem{prop}[thm]{Proposition}

\DeclareMathOperator{\kvech}{\mathrm{vech}}

\DeclareMathOperator{\E}{\mathbb{E}}
\DeclareMathOperator{\var}{\mathbb{V}\mathrm{ar}}

\def\E{\mathbb{E}}

\def\N{\mathbb{N}}
\def\P{\mathbb{P}}
\def\R{\mathbb{R}}
\def\Z{\mathbb{Z}}

\makeatother

\usepackage{babel}
  \providecommand{\examplename}{Example}
  \providecommand{\remarkname}{Remark}
\providecommand{\theoremname}{Theorem}
\providecommand{\lemmaname}{Lemma}

\begin{document}

\title{On the tail behavior of a class of multivariate conditionally heteroskedastic processes}

\author{Rasmus Søndergaard Pedersen\thanks{\noindent Department of Economics, University of Copenhagen. Øster Farimagsgade 5, DK-1353 Copenhagen K, Denmark. $^\dagger$Department of Mathematical Sciences, University of Copenhagen, Universitetsparken 5, DK-2100 Copenhagen Ø, Denmark \& Sorbonne Universités, UPMC Univ. Paris 06, LSTA, Case 158 4 Place Jussieu,
75005 Paris, France. \newline We are grateful for comments and suggestions from the editor-in-chief (Thomas Mikosch), an associate editor, and two referees, which have led to a much improved manuscript. Moreover, we thank Sebastian Mentemeier for valuable comments. Pedersen greatly acknowledges funding from the Carlsberg Foundation. Financial support by the ANR network AMERISKA ANR 14 CE20 0006 01 is gratefully acknowledged by Wintenberger. Correspondence to: Rasmus S. Pedersen, {\tt rsp@econ.ku.dk}.}  \ and Olivier Wintenberger$^\dagger$}

\date{October 20, 2017}

\maketitle

\begin{abstract}
  Conditions for geometric ergodicity of multivariate autoregressive conditional heteroskedasticity (ARCH) processes, with the so-called BEKK (Baba, Engle, Kraft, and Kroner) parametrization, are considered. We show for a class of BEKK-ARCH processes that the invariant distribution is regularly varying. In order to account for the possibility of different tail indices of the marginals, we consider the notion of vector scaling regular variation (VSRV), closely related to non-standard regular variation. The characterization of the tail behavior of the processes is used for deriving the asymptotic properties of the sample covariance matrices.
\end{abstract}
\vspace{7pt}

\emph{AMS 2010 subject classifications:} 60G70, 60G10, 60H25, 39A50.
\vspace{3pt}

\emph{Keywords and phrases:} Stochastic recurrence equations, Markov processes, regular variation, multivariate ARCH, asymptotic properties, geometric ergodicity.

\section{Introduction} \label{sec:intro}
The aim of this paper is to investigate the tail behavior of a class of multivariate conditionally heteroskedastic processes. Specifically, we consider the BEKK-ARCH (or BEKK(1,0,$l$)) process, introduced by \cite{EngleKroner1995}, satisfying
\begin{eqnarray}
X_{t} & = & H_{t}^{1/2}Z_{t},\quad t\in\mathbb{N}\label{eq:BEKK1}\\
H_{t} &= & C+\sum_{i=1}^{l}A_{i}X_{t-1}X_{t-1}^\intercal A_{i}^\intercal,\label{eq:BEKK2}
\end{eqnarray}
with $(Z_{t}:t\in\mathbb{N})$ $i.i.d.$, $Z_t \sim N(0,I_d)$, $C$ a $d \times d$ positive definite matrix, $A_1,...,A_l\in M(d,\mathbb{R})$ (the set of $d \times d$ real matrices), and some initial value $X_0$. Due to the assumption that $Z_t$ is Gaussian, it holds that $X_t$ can be written as the stochastic recurrence equation (SRE)
\begin{eqnarray}
X_{t}=\tilde{M}_{t}X_{t-1}+Q_{t} \label{eq:BEKK3},
\end{eqnarray}
with
\begin{eqnarray}
\tilde{M}_{t}=\sum_{i=1}^{l}m_{it}A_{i} \label{eq:BEKK4}
\end{eqnarray}
and $(m_{it}:t\in \mathbb{N})$ is an $i.i.d.$ process mutually independent of $(m_{jt}:t\in \mathbb{N})$ for $i\neq j$, with $m_{it}\sim N(0,1)$. Moreover $(Q_{t}:t\in \mathbb{N})$ is an $i.i.d.$ process with $Q_{t}\sim N(0,C)$ mutually independent of $(m_{it}:t\in \mathbb{N})$ for all $i=1,...,l$.

To our knowledge, the representation in \eqref{eq:BEKK3}-\eqref{eq:BEKK4} of the BEKK-ARCH process is new. Moreover, the representation will be crucial for studying the stochastic properties of the process. Firstly, we find a new sufficient condition in terms of the matrices $A_1,...,A_l$ in order for  $(X_t:t\geq 0)$ to be geometrically ergodic. In particular, for the case $l=1$, we derive a condition directly related to the eigenvalues of $A_1$, in line with the strict stationarity condition found by \cite{Nelson1990StationarityGARCH} for the univariate ARCH(1) process. This condition is milder compared to the conditions found in the existing body of literature on BEKK-type processes. Secondly, the representation is used to characterize the tails of the stationary solution to $(X_t:t\in \N)$.

Whereas the tail behavior of univariate GARCH processes is well-established, see e.g. \cite{Basrak2002RegularVariation}, few results on the tail behavior of multivariate GARCH processes exist. Some exceptions are the multivariate constant conditional correlation (CCC) GARCH processes, see e.g. \cite{Starica1999}, \cite{Pedersen2015targeting}, and \cite{Matsui2016extremogram}, and a class of factor GARCH processes, see \cite{Basrak2009RVMultivariate}. This existing body of literature relies on rewriting the (transformed) process on companion form that obeys a non-negative multivariate SRE. The characterization of the tails of the processes then follows by an application of Kesten's Theorem (\cite{Kesten1973randomcoef}) for non-negative SREs. Such approach is not feasible when analyzing BEKK-ARCH processes, as these are stated in terms of an $\R^d$-valued SRE in \eqref{eq:BEKK3}.
For some special cases of the BEKK-ARCH process, we apply existing results for $\R^d$-valued SREs in order to  show that the stationary distribution for the BEKK-ARCH process is multivariate regularly varying. Specifically, when the matrix $\tilde{M}_{t}$ in \eqref{eq:BEKK4} is invertible (almost surely) and has a law that is absolutely continuous with respect to the Lebesgue measure on $M(d,\mathbb{R})$ (denoted ID BEKK-ARCH) we argue that the classical results of Kesten (1973, Theorem 6), see also \cite{alsmeyer2012sre}, apply. Moreover, when $\tilde{M}_{t}$ is the product of a positive scalar and a random orthogonal matrix (denoted Similarity BEKK-ARCH) we show that the results of  \cite{Buraczewski2009} apply. Importantly, we do also argue that the results of \cite{alsmeyer2012sre} rely on rather restrictive conditions that can be shown not to hold for certain types of BEKK-ARCH processes, in particular the much applied process where $l=1$ and $A_1$ is diagonal, denoted Diagonal BEKK-ARCH.  Specifically, and as ruled out in \cite{alsmeyer2012sre}, we show that the Diagonal BEKK-ARCH process exhibits different marginal tail indices, i.e. $\P(\pm X_{t,i}>x)/ c_ix^{-\alpha_i} \to 1$ as $x\to \infty$ for some constant $c_i>0$, $i=1,...,d$ (denoted Condition {\bf M}).  In order to analyze this class of BEKK-ARCH processes, where the tail indices are allowed to differ among the elements of $X_t$, we introduce a new notion of vector scaling regular variation (VSRV) distributions, based on element-wise scaling of $X_t$ instead of scaling by an arbitrary norm of $X_t$. We emphasize that the notion of VSRV is similar to the notion of non-standard regular variation (see Resnick (2007, Chapter 6)) under the additional Condition {\bf M}. In addition, in the spirit of \cite{Basrak2009RVMultivariate}, we introduce the notion of VSRV processes with particular attention to Markov chains and characterize their extremal behavior. We argue that the stationary distribution of  the Diagonal BEKK-ARCH process is expected to be VSRV, which is supported in a simulation study. Proving that the VSRV property applies requires that new multivariate renewal theory is developed, and we leave such task for future research.

The rest of the paper is organized as follows. In Section \ref{sec:geo_erg}, we state sufficient conditions for geometric ergodicity of the BEKK-ARCH process and introduce the notion of vector-scaling regular varying (VSRV) distributions. We show that the distribution of $X_t$ satisfies this type of tail-behavior, under suitable conditions. In Section \ref{sec:srvtimeseries} we introduce the notion of VSRV processes and state that certain BEKK-ARCH processes satisfy this property. Moreover, we consider the extremal behavior of the process, in terms of  the asymptotic behavior of maxima and extremal indices. Lastly, we consider the convergence of point processes based on VSRV processes. In Section \ref{sec:autocov}, we consider the limiting distribution of the sample covariance matrix of $X_t$, which relies on point process convergence. Section \ref{sec:conclusion} contains some concluding remarks on future research directions.

Notation: Let $GL(d,\mathbb{R})$ denote the set of $d\times d$ invertible real matrices. With $M(d,\mathbb{R})$ the set of $d \times d$ real matrices and $A\in M(d,\mathbb{R})$, let $\rho(A)$ denote the spectral radius of $A$. With $\otimes$ denoting the Kronecker product, for any real matrix $A$ let $A^{\otimes p}= A\otimes A\otimes\cdots\otimes A$ ($p$ factors). For two matrices, $A$ and $B$, of the same dimension, $A\odot B$ denotes the elementwise product of $A$ and $B$. Unless stated otherwise, $\Vert \cdot \Vert $ denotes an arbitrary matrix norm. Moreover, $\mathbb{S}^{d-1}=\{x\in\mathbb{R}^d:\Vert x \Vert =1\}$. For two matrices $A$ and $B$ of the same dimensions, $A \nleqslant B$ means that $A_{ij} > B_{ij}$ for some $i,j$.  For two positive functions $f$ and $g$, $f(x) \sim g(x)$, if $\lim_{x\rightarrow\infty}f(x)/g(x)=1$. Let $\mathcal L(X)$ denote the distribution of $X$. By default, the mode of convergence for distributions is weak convergence.

\section{Stationary solution of the BEKK-ARCH model} \label{sec:geo_erg}

\subsection{Existence and geometric ergodicity}
We start out by stating the following theorem that provides a sufficient condition for geometric ergodicity of the BEKK-ARCH process. To our knowledge, this result together with Proposition \ref{prop:simple_geo_erg} below are new.

\begin{thm}
\label{thm:geo_erg}Let $X_t$ satisfy \eqref{eq:BEKK1}-\eqref{eq:BEKK2}. With $\tilde{M}_t$ defined in \eqref{eq:BEKK4},  suppose that
\begin{equation}
\inf_{n\in\mathbb{N}}\left\{ \frac{1}{n}\E\left[\log\left(\left\Vert \prod_{t=1}^{n}\tilde{M}_t\right\Vert \right)\right]\right\} <0.\label{eq:lyapunov_cond}
\end{equation}
Then $(X_{t}:t=0,1,...)$ is geometrically ergodic, and for the associated
stationary solution, $\E[\Vert X_{t}\Vert^{s}]<\infty$ for some $s>0$.
\end{thm}
The proof of the theorem follows by \cite[Theorems 2.1-2.2, Example 2.6.d, and Theorem 3.2]{alsmeyer:2003} and is hence omitted.
%
\begin{rem} \label{rem:moments}
  A sufficient condition for the existence of finite higher-order moments of $X_t$ can be obtained from Theorem 5 of \cite{feigin1985random}. In particular, if $\rho(E[\tilde M_t^{\otimes 2n}]) < 1$ for some $n\in \mathbb{N}$,
  then, for the strictly stationary solution, $E[\Vert X_t \Vert^{2n}]<\infty$. For example, $\rho(\sum_{i=1}^{l}A_i^{\otimes 2})<1$ implies that $E[\Vert X_t \Vert^{2}]<\infty$. This result complements Theorem C.1 of \cite{Pedersen2013VTBEKK} that contains conditions for finite higher-order moments for the case $l=1$.
\end{rem}

For the case where $\tilde{M}_t$ contains only one term, i.e. $l=1$, the condition in \eqref{eq:lyapunov_cond} simplifies and a condition for geometric ergodicity can be stated explicitly in terms of the eigenvalues of the matrix $A_1$:

\begin{prop}
\label{prop:simple_geo_erg} Let $X_t$ satisfy \eqref{eq:BEKK1}-\eqref{eq:BEKK2} with $l=1$ and let $A:=A_1$. Then a necessary and sufficient condition for \eqref{eq:lyapunov_cond} is that
\begin{equation}
\rho(A)<\exp\left\{ \frac{1}{2}\left[-\psi(1)+\log(2)\right]\right\} = 1.88736...,\label{eq:cond_spectral_rad_1}
\end{equation}
where $\psi(\cdot)$ is the digamma function.
\end{prop}

\begin{proof}
The condition \eqref{eq:lyapunov_cond} holds if and only if there exists $n\in\mathbb{N}$ such that
\begin{equation}
\E\left[\log\left(\left\Vert \prod_{t=1}^{n}\tilde M_{t}\right\Vert \right)\right]<0.\label{eq:suff_cond_1}
\end{equation}
 Let $m_{t} := m_{1t}$. It holds that
\begin{eqnarray*}
\E\left[\log\left(\left\Vert \prod_{t=1}^{n}\tilde M_{t}\right\Vert \right)\right] & = & \E\left[\log\left(\left\Vert A^{n}\prod_{t=1}^{n}m_{t}\right\Vert \right)\right]\\
 & = & \log\left(\left\Vert A^{n}\right\Vert \right)-n\E\left[-\log(|m_{t}|)\right]\\
 & = & \log\left(\left\Vert A^{n}\right\Vert \right)-n\left\{ \frac{1}{2}\left[-\psi(1)+\log(2)\right]\right\} ,
\end{eqnarray*}
and hence \eqref{eq:suff_cond_1} is satisfied if
\[
\log\left(\left\Vert A^{n}\right\Vert ^{1/n}\right)<\frac{1}{2}\left[-\psi(1)+\log(2)\right].
\]
The result now follows by observing that $\left\Vert A^{n}\right\Vert ^{1/n}\rightarrow\rho(A)$
as $n\rightarrow\infty$.
\end{proof}

\begin{rem}\label{rem:spec_rad}
It holds that $\rho(A^{\otimes 2})=(\rho(A))^{2}$. Hence the condition in \eqref{eq:cond_spectral_rad_1} is equivalent to
\begin{eqnarray*}
\rho(A^{\otimes 2}) & < & \exp\left\{ -\psi(1)+\log(2)\right\} = \frac{1}{2}\exp\left[-\psi\left(\frac{1}{2}\right)\right]=3.56...,
\end{eqnarray*}
which is similar to the strict stationary condition found for the
ARCH coefficient of the univariate ARCH(1) process with
Gaussian innovations; see \citet{Nelson1990StationarityGARCH}.\\
\cite{Boussama2011} derive sufficient conditions for geometric ergodicity of the GARCH-type BEKK process, where $H_t = C + \sum_{i=1}^{p}{A}_i X_{t-i}X_{t-i}^\intercal{A}_i^\intercal + \sum_{j=1}^{q}{B}_j H_{t-j}{B}_j^\intercal$, $A_i,B_j\in M(d,\R), i=1,...,p, j=1,...,q$. Specifically, they show that a sufficient condition is $\rho(\sum_{i=1}^{p}{A}_i^{\otimes 2}+\sum_{j=1}^{q}{B}_j^{\otimes 2})<1$. Setting $p=1$ and $q=0$, this condition simplifies to $\rho({A}_1^{\otimes 2})<1$, which is stronger than the condition derived in \eqref{eq:cond_spectral_rad_1}.
\end{rem}

Below, we provide some examples of BEKK-ARCH processes that are geometrically ergodic and that will be studied in detail throughout this paper.

\begin{example}[\textbf{ID BEKK-ARCH}] \label{ex:id}
 Following \citet{alsmeyer2012sre}, we consider BEKK processes with corresponding SRE's satisfying certain irreducibility and density conditions (ID), that is conditions {\bf (A4)}-{\bf (A5)} in Section \ref{sec:Alsmeyer} in the appendix. Specifically, we consider the \emph{bivariate} BEKK-ARCH process in \eqref{eq:BEKK1}-\eqref{eq:BEKK2} with
 \begin{eqnarray*}
H_{t} &= & C+\sum_{i=1}^{4}A_{i}X_{t-1}X_{t-1}^\intercal A_{i}^\intercal,
\end{eqnarray*}
where
\begin{eqnarray}
A_1 = \left(
  \begin{array}{cc}
    a_1 & 0 \\
    0 & 0 \\
  \end{array}
\right)
\quad A_2 = \left(
  \begin{array}{cc}
    0 & 0 \\
    a_2 & 0 \\
  \end{array}
\right),
\quad A_3 = \left(
  \begin{array}{cc}
    0 & a_3 \\
    0 & 0 \\
  \end{array}
\right),
\quad A_4 = \left(
  \begin{array}{cc}
    0 & 0 \\
    0 & a_4 \\
  \end{array}
\right) \label{eq:id3}
\end{eqnarray}
and
\begin{eqnarray}
  a_1,a_2,a_3,a_4\neq 0. \label{eq:id4}
\end{eqnarray}
Writing $X_t$ as an SRE, we obtain
\begin{eqnarray}
X_{t}=\tilde{M}_{t}X_{t-1}+Q_{t}, \label{eq:id1}
\end{eqnarray}
with
\begin{eqnarray}
  \tilde{M}_t = \sum_{i=1}^4 A_i m_{it} \label{eq:id2}
\end{eqnarray}
where $(m_{1t}),(m_{2t}),(m_{3t}),(m_{4t})$ are mutually independent i.i.d. processes with $m_{it}\sim N(0,1)$. Assuming that $a_1,a_2,a_3,a_4$  are such that the top Lyapunov exponent of $(\tilde{M}_t)$ is strictly negative, we have that the process is geometrically ergodic.

Notice that one could consider a more general $d$-dimensional process with the same structure as in \eqref{eq:id3}-\eqref{eq:id2}, but with $\tilde M_t$ containing $d^2$ terms such that $\tilde M_t$  has a Lebesgue density on $M(d,\mathbb{R})$, as clarified in Example \ref{ex:id2} below. Moreover, one could include additional terms to $\tilde{M}_t$, say a term containing a full matrix $A$ or an autoregressive term, as presented in Remark \ref{rem:arbekk} below. We will focus on the simple bivariate process, but emphasize that our results apply to more general processes. $\square$
\end{example}

\begin{example}[\textbf{Similarity BEKK-ARCH}] \label{ex:similarity}
  Consider the BEKK process in \eqref{eq:BEKK1}-\eqref{eq:BEKK2} with $l=1$ and $A:=A_{1}=aO$, where $a$ is a positive scalar and $O$ is an orthogonal matrix. This implies that the SRE \eqref{eq:BEKK3} has $\tilde{M}_t=am_tO$. By definition, $\tilde{M}_t$  is a \emph{similarity} with probability one, where we recall that a matrix is a similarity if it can be written as a product of a positive scalar and an orthogonal matrix. From Proposition \ref{prop:simple_geo_erg}, we have that if $a<\exp\left\{ (1/2)\left[-\psi(1)+\log(2)\right]\right\} = 1.88736...$, then the process is geometrically ergodic.
  An important process satisfying the similarity property is the well-known scalar BEKK-ARCH process, where $H_t=C+aX_{t-1}X_{t-1}^\intercal$, $a>0$. Here $A=\sqrt{a}I_{d}$, with $I_{d}$ the identity matrix. $\square$
\end{example}

\begin{example}[\textbf{Diagonal BEKK-ARCH}] \label{ex:diagonal}
Consider the BEKK-ARCH process in \eqref{eq:BEKK1}-\eqref{eq:BEKK2} with $l=1$ such that $A:=A_1$ is diagonal. We refer to this process as the Diagonal BEKK-ARCH process. Relying on Proposition \ref{prop:simple_geo_erg}, the process is geometrically ergodic, if each diagonal element of $A$ is less than $\exp\left\{ (1/2)\left[-\psi(1)+\log(2)\right]\right\} = 1.88736...$ in modulus.

As discussed in  \cite{BLR2006}, diagonal BEKK models are typically used in practice, e.g. within empirical finance, due to their relatively simple parametrization. As will be shown below, even though the parametrization is simple, the tail behavior is rather rich in the sense that each marginal of $X_t$ has different tail indices, in general. $\square$
\end{example}

\begin{rem}\label{rem:arbekk}
As an extension to \eqref{eq:BEKK1}-\eqref{eq:BEKK2}, one may consider the autoregressive BEKK-ARCH (AR BEKK-ARCH) process
\begin{eqnarray*}
X_{t} &= & A_0 X_{t-1} + H_{t}^{1/2}Z_{t},\quad t\in\mathbb{N} \\
H_{t} &= & C+ \sum_{i=1}^{l}A_{i}X_{t-1}X_{t-1}^\intercal A_{i}^\intercal,
\end{eqnarray*}
with $A_0 \in M(\mathbb{R},d)$. This process has recently been studied and applied by \cite{Nielsen2014vectorDAR} for modelling the term structure of interest rates. Notice that the process has the SRE representation
\begin{eqnarray*}
X_{t}=\tilde{M}_{t}X_{t-1}+Q_{t}, \quad \tilde{M}_{t}= A_0 + \sum_{i=1}^{l}m_{it}A_{i}.
\end{eqnarray*}
Following the arguments used for proving Theorem \ref{thm:geo_erg}, it holds that the AR BEKK-ARCH process is geometrically ergodic if condition \eqref{eq:lyapunov_cond} is satisfied. Interestingly, as verified by simulations in \cite{Nielsen2014vectorDAR} the Lyapunov condition may hold even if the autoregressive polynomial has unit roots, i.e. if $A_0 = I_{d} + \Pi$, where $\Pi\in M(\mathbb{R},d)$ has reduced rank.
\end{rem}

\subsection{Multivariate regularly varying distributions}

The stationary solution of the BEKK-ARCH process (see Theorem \ref{thm:geo_erg}) can be written as
\begin{eqnarray}
X_{t}  =  \sum_{i=0}^{\infty}\prod_{j=1}^{i}\tilde{M}_{t-j+1}Q_{t-i}, \quad t\in \Z.\label{eq:stationary_solution}
\end{eqnarray}
Even if the random matrices $\tilde{M}_{t}$ are light-tailed under the Gaussian assumption,
the maximum of the  products $(\prod_{t=1}^{T}\tilde{M}_{t})_{T\ge 0}$ may exhibit heavy tails when $T\to \infty$. More precisely, the tails of the stationary distribution are suspected to have an extremal behavior as a power law function: For any $u\in \mathbb{S}^{d-1}$,
\begin{eqnarray}
  \P(u^\intercal X_0>x)\sim C(u)x^{-\alpha(u)},\qquad x\to\infty, \label{eqn:rv1}
\end{eqnarray}
with $\alpha(u)>0$ and $C(u_0)>0$ for some $u_0 \in \mathbb{S}^{d-1}$. The cases where $\alpha(u)=\alpha$ and $C(u)>0$ for all $u \in \mathbb{S}^{d-1}$ are referred as  Kesten's cases, because of the seminal paper \cite{Kesten1973randomcoef}, and are the subject of the monograph by \cite{buraczewski2016stochastic}. A class of multivariate distributions satisfying this property is the class of multivariate regularly varying distributions (\cite{dehaan:resnick:1977}):
\begin{dfn}\label{def:MRV}
  Let $\bar{\mathbb{R}}^d_0:=\bar{\mathbb{R}}^d\setminus\{0\}$, $\bar{\mathbb{R}}:=\mathbb{R}\cup \{-\infty,\infty\}$, and $\bar{\mathcal{B}}_0^d$ be the Borel $\sigma$-field of $\bar{\mathbb{R}}^d_0$. For an $\mathbb{R}^d$-valued random variable $X$ and some constant scalar $x>0$, define $\mu_x(\cdot):=\mathbb{P}(x^{-1}X\in \cdot)/\mathbb{P}(\Vert X\Vert > x)$. Then $X$ and its distribution are \emph{multivariate regularly varying }if there exists a non-null Radon measure $\mu$ on $\bar{\mathcal{B}}_0^d$ which satisfies
  \begin{eqnarray}\label{eq:mrv}
    \mu_x(\cdot) \to \mu(\cdot) \qquad \text{vaguely, as $x\to \infty$.}
  \end{eqnarray}
  For any $\mu$-continuity set $C$ and $t>0$, $\mu(tC)=t^{-\alpha}\mu(C)$, and we refer to $\alpha$ as the index of regular variation.
\end{dfn}
We refer to  \cite{dehaan:resnick:1977} for the notion of vague convergence and additional details. Below, we provide two examples of multivariate regularly varying BEKK processes.

\begin{example}[\textbf{ID BEKK-ARCH}, continued] \label{ex:id2}
Consider the ID BEKK-ARCH process \eqref{eq:id3}-\eqref{eq:id2} from Example \ref{ex:id}. By verifying conditions {\bf (A1)}-{\bf (A7)} of Theorem 1.1 of  \citet{alsmeyer2012sre}, stated in Section \ref{sec:Alsmeyer} in the appendix, we establish that the process is multivariate regularly varying.

Since $(m_{1t},m_{2t},m_{3t},m_{4t})$ and $Q_t$ are Gaussian, we have that {\bf (A1)}-{\bf (A2)} hold. Moreover,
\begin{eqnarray} \label{eq:M_id}
\tilde{M_t} = \left(
  \begin{array}{cc}
    a_1m_{1t} & a_3m_{3t} \\
     a_2m_{2t} &  a_4m_{4t} \\
  \end{array}
\right)
\end{eqnarray}
is invertible with probability one, which ensures that {\bf (A3)} is satisfied. From \eqref{eq:M_id} we also notice that the distribution of $\tilde M_t$ has a Lebesgue density on $M(d,\mathbb{R})$ which is strictly positive in a neighborhood of $I_2$. This ensures that the irreducibility and density conditions {\bf (A4)}-{\bf (A5)} are satisfied. The fact that $Q_t \sim N(0,C)$ and independent of $\tilde{M}_t$ implies that condition {\bf (A6)} holds. Lastly, condition {\bf (A7)} holds by the fact that $(m_{1t},m_{2t},m_{3t},m_{4t})$ and $Q_t$ are Gaussian. By Theorem 1.1 of \cite{alsmeyer2012sre} we have established the following proposition:
\begin{prop}
Let $X_t$ satisfy \eqref{eq:id3}-\eqref{eq:id2} such that the top Lyapunov exponent of $(\tilde{M}_t)$ is strictly negative. Then for the stationary solution $(X_t)$, there exists $\alpha>0$ such that
\begin{eqnarray}
  \lim_{t\rightarrow\infty} t^{\alpha}\mathbb{P}(x^\intercal X_0>t) = K(x), \quad x \in \mathbb{S}^{1},
\end{eqnarray}
for some finite, positive, and continuous function $K$ on $\mathbb{S}^{1}$.
\end{prop}

The proposition implies that each marginal of the distribution of $X_0$ is regularly varying of order $\alpha$. By Theorem 1.1.(ii) of \cite{Basrak2002Characterization}, we conclude that $X_0$ is multivariate regularly varying whenever $\alpha$ is a non-integer. Moreover, since $X_0$ is symmetric, the multivariate regular variation does also hold if $\alpha$ is an odd integer, see Remark 4.4.17 in \cite{buraczewski2016stochastic}.

The proposition does also apply if $a_1=0$ or $a_4=0$. This can be seen by observing that $\prod_{k=1}^{n}\tilde{M}_{k}$ has a strictly positive density on $M(d,\mathbb{R})$ for $n$ sufficiently large, which is sufficient for establishing conditions {\bf (A4)}-{\bf (A5)}. $\square$
\end{example}

\begin{example}[\textbf{Similarity BEKK-ARCH}, continued] \label{ex:similarity2}
  The Similarity BEKK-ARCH, introduced in Example \ref{ex:similarity}, fits into the setting of \cite{Buraczewski2009}, see also Section 4.4.10 of \cite{buraczewski2016stochastic}. Specifically, using the representation $\tilde{M_t}=a|m_t|\mathrm{sign}(m_t)O$, we have that
   \begin{enumerate}[(i)]
     \item $\E[\log(|m_ta|)]<0$ if $a<\exp\left\{ (1/2)\left[-\psi(1)+\log(2)\right]\right\}$,
     \item $\P(\tilde{M_t}x+Q_t=x)<1$ for any $x\in \mathbb{R}^{d}$, and
     \item $\log(|am_t|)$ has a non-arithmetic distribution.
   \end{enumerate}
   Then, due to Theorem 1.6 of \cite{Buraczewski2009}, we have the following proposition:
\begin{prop}
\label{thm:similarities}Let $X_t$ satisfy \eqref{eq:BEKK1}-\eqref{eq:BEKK2} with $l=1$ such that $A:=A_{1}=aO$, where $a>0$ and $O$ is an orthogonal matrix. If $a<\exp\left\{ (1/2)\left[-\psi(1)+\log(2)\right]\right\} = 1.88736...$, then the process has a unique strictly stationary solution $(X_{t})$ with $X_t$ multivariate regularly varying with index $\alpha>0$ satisfying $E[(|m_t|a)^\alpha]=1$. $\square$
\end{prop}
\end{example}

In the following example, we clarify that the Diagonal BEKK-ARCH process, introduced in Example \ref{ex:diagonal}, does not satisfy the conditions of Theorem 1.1 of \cite{alsmeyer2012sre}. Moreover, we argue that the marginals may have different tail indices, which motivates the notion of \emph{vector scaling regular variation}, introduced in the next section.

\begin{example}[\textbf{Diagonal BEKK-ARCH}, continued] \label{ex:diagonal2}
Consider the diagonal BEKK-ARCH process in Example \ref{ex:diagonal}, i.e. \eqref{eq:BEKK1}-\eqref{eq:BEKK2} with $l=1$ such that $A:=A_1$ is diagonal, $m_t:=m_{1t}$, and $M_t : =\tilde{M}_t = m_tA$. For this process, the distribution of $M_{t}$ is too restricted to apply the results by \citet{alsmeyer2012sre}, as in Example \ref{ex:id2}. Specifically, the irreducibility condition {\bf (A4)} in Appendix \ref{sec:Alsmeyer} can be shown not to hold, as clarified next. It holds that
\begin{eqnarray*}
\mathbb{P}\left\{ \Vert x^\intercal \prod_{k=1}^{n}M_{k}\Vert^{-1}\left(x^\intercal \prod_{k=1}^{n}M_{k}\right)\in U\right\}  & = & \mathbb{P}\left\{ |\prod_{k=1}^{n}m_{k}|^{-1}\Vert x^\intercal A^{n}\Vert^{-1}\left(\prod_{k=1}^{n}m_{k}\right)x^\intercal A^{n}\in U\right\} \\
 & = & \mathbb{P}\left\{ \mathrm{sign}\left(\prod_{k=1}^{n}m_{k}\right)\Vert x^\intercal A^{n}\Vert^{-1}x^\intercal A^{n}\in U\right\}.
\end{eqnarray*}
Hence for any $x\in\mathbb{S}^{d-1}$ we can always find a non-empty open $U\subset \mathbb{S}^{d-1}$ such that
\begin{equation}
\max_{n\in\mathbb{N}}\mathbb{P}\left\{ \mathrm{sign}\left(\prod_{k=1}^{n}m_{k}\right)\Vert x^\intercal A^{n}\Vert^{-1}x^\intercal A^{n}\in U\right\} =0.\label{eq:max_P_1}
\end{equation}
As an example, for $d=2$, choose $x=(1,0)^\intercal$. Then
$\Vert x^\intercal A^{n}\Vert^{-1} x^\intercal A^{n}\in\{(-1,0)\}\cup\{(1,0)\}$
for any $n\in\mathbb{N}$. We conclude that condition {\bf (A4)} does not hold for the diagonal BEKK-ARCH process.\

Note that, each element of $X_{t}=(X_{t,1},...,X_{t,d})^{\intercal}$ of the diagonal BEKK-ARCH process can be written as an SRE,
\[
X_{t,i}=A_{ii}m_{t}X_{t-1,i}+Q_{t,i},\qquad t\in Z,\qquad i=1,\ldots,d.
\]
By Theorem 4.1 of \cite{goldie1991}, the stationary solution of the marginal equation exists if and only if $\E[\log(|A_{ii}m_0|)]<0$.
In that case there exists a unique $\alpha_i>0$ such that $\E[|m_0|^{\alpha_i}]=|A_{ii}|^{-{\alpha_i}}$ and
\[
\P(\pm X_{0,i}>x)\sim c_ix^{-\alpha_i}\qquad\mbox{where}\qquad c_i=\frac{\E[|X_{1,i}|^{\alpha_i}-|A_{ii}m_{1 }X_{0,i}|^{\alpha_i}]}{2\alpha_i \E[|A_{ii}m_{1}|^{\alpha_i}\log(|A_{ii}m_{1}|)]}.
\]
Hence each marginal of $X_0$ may in general have different tail indices. More precisely, the tail indices are different if the diagonal elements of $A$, i.e. the $A_{ii}$s, are, and the heaviest marginal tail index  $\alpha_{i_0}$ corresponds to the largest diagonal coefficient $A_{i_0i_0}$. When $i_0$ is unique, i.e. $\alpha_{i_0}<\alpha_{i}$ for all $i=1,...,d$ except $i\neq i_0$, the distribution $X_0$ can   be considered as  multivariate regularly varying with index $\alpha_{i_0}$ and with a limit measure $\mu$ with degenerate marginals $i\neq i_0$.$\square$
\end{example}

\subsection{Vector scaling regularly varying distributions}\label{sec:vs}
The previous Example \ref{ex:diagonal2} shows that the Diagonal BEKK-ARCH process fits into the case where $\alpha(u)$ in \eqref{eqn:rv1} is non-constant. Such cases have not attracted much attention in the existing body of literature. However, recent empirical studies, such as \cite{Matsui2016extremogram}, see also \cite{damek2017}, may suggest that it is more realistic to consider different marginal tail behaviors when modelling multidimensional financial observations. The idea is to use a vector scaling instead of the scaling $\P(\|X\|>x)$ in Definition \ref{def:MRV} that reduced the regular variation properties of the vector $X$ to the regular variation properties of the norm $\|X\|$ only. More precisely, let $(X_t)$ be a stationary process in $\mathbb{R}^d$ and let $x=(x_1,\ldots,x_d)^\intercal\in \mathbb{R}^d$. Denote also $x^{-1}=(x^{-1}_1,\ldots,x_d^{-1})^\intercal$.\\

In our framework, we consider distributions satisfying the following condition:
\begin{description}
\item[Condition M] Each marginal of $X_0$ is regularly varying of order $\alpha_i>0, i=1,...,d$. The slowly varying functions $\ell_i(t)\to c_i>0$ as $t\to \infty$, $i=1,...,d$.
\end{description}
Indeed, the Diagonal BEKK-ARCH process introduced in Example \ref{ex:diagonal2} satisfies Condition {\bf M}. Moreover, any regularly varying distribution satisfying the Kesten property \eqref{eqn:rv1} satisfies Condition {\bf M}. In particular, the ID and Similarity BEKK-ARCH processes, introduced in Examples \ref{ex:id} and \ref{ex:similarity} respectively, satisfy Condition {\bf M}.

We introduce the notion of vector scaling regular variation as the nonstandard regular variation of the book of \cite{resnick:2007} under Condition {\bf M}, extended to  negative components (Resnick, 2007, Sections 6.5.5-6.5.6):
\begin{dfn}\label{def:gen}
The distribution of  the vector $X_0$   is vector scaling regularly varying (VSRV) if and only if it satisfies Condition {\bf M} and it is non-standard regularly varying, i.e. there exists a normalizing sequence $x(t)$ and a Radon measure $\mu$ with  non-null marginals such that
\begin{eqnarray} \label{eq:srv1}
t \P(x(t)^{-1}\odot X_0 \in \cdot) \to \mu(\cdot),\qquad\mbox{vaguely}.
\end{eqnarray}
\end{dfn}
The usual way of analyzing non-standard regularly varying vectors is to consider a componentwise normalization that is standard regularly varying in the sense of Definition \ref{def:MRV}. Specifically, when $X_0=(X_{0,1},...,X_{0,d})^\intercal$ satisfies Definition \ref{def:gen}, $(c_1^{-1}(X_{0,1}/|X_{0,1}|)|X_{0,1}|^{\alpha_1},...,c_d^{-1}(X_{0,d}/|X_{0,d}|)|X_{0,d}|^{\alpha_d})^\intercal$ satisfies Definition \ref{def:MRV} with index one. Throughout we find it helpful to focus on the non-normalized vector $X_0$ in order to preserve the multiplicative structure of the tail chain introduced in Section \ref{sec:tailchain} below, which is used for analyzing the extremal properties of VSRV processes.

In the following proposition we state the VSRV vector $X_0$ has a polar decomposition. In the case where Condition {\bf M} is not satisfied, note that the polar decomposition holds on a transformation of the original process. Under Condition {\bf M}, the natural radius notion is $\|\cdot\|_\alpha$, where
\begin{eqnarray}
  \|x\|_{\alpha}:=\max_{1\le i\le d}c_i^{-1}|x_i|^{\alpha_i}. \label{eq:hom_func}
\end{eqnarray}Notice that the homogeneity of  $ \|\cdot \|_{\alpha}$, due to Condition {\bf M}, will be essential for the proof.
\begin{prop}\label{prop:srv}
Suppose that the vector $X_0$ satisfies Condition {\bf M}. Then $X_0$ is VSRV if and only if  there exists a tail vector $Y_0\in \R^d$ with non-degenerate marginals such that  \begin{equation}\label{eq:srv2}
\mathcal L(((c_it)^{-1/\alpha_i})_{1\le i\le d}  \odot X_{0}\mid \|X_{0}\|_\alpha>t) \to_{t\to \infty}\mathcal L(Y_{0}),
\end{equation}
where $\|\cdot\|_\alpha$ is defined in \eqref{eq:hom_func}.
Moreover, $\|Y_0\|_\alpha$ is standard Pareto distributed.
\end{prop}
Notice that a similar vector scaling argument has been introduced in \cite{lindskog:2014}.
\begin{proof}
Adapting Theorem 4 of \cite{dehaan:resnick:1977}, the definition of vector scaling regularly varying distribution of $X_0$ in \eqref{eq:srv1} implies \eqref{eq:srv2}.
Conversely, under Condition {\bf M}, we have that $|X_{0,k}|^{\alpha_k}$ is regularly varying of order 1 for all $1\le k\le d$ with slowly varying functions $\ell_i(t)\sim c_i$. Moreover $\|X_0\|_\alpha$  is regularly varying from the weak convergence in \eqref{eq:srv2} applied on the Borel sets $\{\|X_{0}\|_\alpha>ty\}$, $ y\ge 1$. Thus,  $\|X_0\|_\alpha$  is regularly varying of order 1 with slowly varying function $\ell(t) $. One can rewrite \eqref{eq:srv2} as
\[
\ell(t)^{-1}t  \P(x(t)^{-1}\odot X_0 \in \cdot,\|X_0\|_\alpha>t)  \to \P(Y_0\in \cdot).
\]
Using the slowly varying property of $\ell$, we obtain, for any $\epsilon>0$,
\[
\ell(t)^{-1}t  \P(x(t)^{-1}\odot X_0 \in \cdot,\|X_0\|_\alpha>t\epsilon)  \to \epsilon^{-1}\P(Y_0\in \cdot).
\]
Then by marginal homogeneity of $\|\cdot\|_\alpha$,
\[
\ell(t)^{-1}t  \P(x(t)^{-1}\odot X_0 \in \cdot,\|x(t)^{-1}\odot X_0\|_\alpha>\epsilon)  \to \epsilon^{-1}\P(Y_0\in \cdot).
\]
Notice that $\ell(t) t^{-1}>0$ is non-increasing as it is the tail of $\| X_0\|_\alpha$. So there exists a change of variable $t=h(t')$ so that $\ell(t)^{-1}t=t'$ and
\[
t' \P(x(h(t'))^{-1}\odot X_0 \in \cdot,\|x(h(t'))^{-1}\odot X_0\|_\alpha>\epsilon) \to \epsilon^{-1}\P(Y_0\in \cdot).
\]
We obtain the existence of $\mu$ for $x'= x \circ h$ in \eqref{eq:srv1} such that $\mu(\cdot,\|x\|_\alpha>\epsilon)=\P(\cdot)$, which is enough to characterize $\mu$ entirely, choosing $\epsilon>0$ arbitrarily small.
\end{proof}

The spectral properties of VSRV $X_0$ can be expressed  in terms of the tail vector $Y_0$. Notice that for any $u\in \{+1,0,-1\}^d$, there exists $c_+(u)\ge0$ satisfying
\[
\lim_{t\to\infty}\P\left(\max_{1\le i\le d}c_i^{-1}(u_iX_{0,i})_+^{\alpha_i}>t\mid \|X_0\|_\alpha>t\right)=c_+(u).
\]
Consider $c^{-1}\odot( u\odot X_0)_+^\alpha$, where $c^{-1}=(c_1^{-1},\ldots,c_d^{-1})^\intercal$ and for $x\in \mathbb{R}^d$ and $\alpha =(\alpha_1,...,\alpha_d)^\intercal$, $(x)_+^\alpha = ( (x_1)_+^{\alpha_1},...,(x_d)_+^{\alpha_d})^\intercal$.
If $c_+(u)$ is non-null, by a continuous mapping argument, $c^{-1}\odot( u\odot X_0)_+^\alpha$ satisfies
\begin{equation}
\label{eq:convu}
\mathcal L(t^{-1}c^{-1}\odot( u\odot X_{0})_+^\alpha\mid \|(u\odot X_{0})_+ \|_\alpha>t) \to_{t\to \infty}\mathcal L(c_+(u)^{-1}(u\odot Y_{0})_+^\alpha),
\end{equation}
and $c^{-1}\odot (u\odot X_{0})_+^\alpha$ is regularly varying of index 1.  By homogeneity of the limiting measure in the multivariate  regular variation \eqref{eq:mrv}, we may decompose the limit as a product
\[
\frac{\P((\|  (u\odot X_{0})_+ \|_\alpha>ty,c^{-1}\odot (u\odot X_{0})_+^\alpha/\|( u\odot X_{0})_+ \|_\alpha\in \cdot)}{\P( \|( u\odot X_{0})_+ \|_\alpha>t)}\to y^{-\alpha}\P_{\Theta_u}(\cdot),
\]
for any $y\ge 1$.
Such limiting distribution is called a simple max-stable distribution, and $\P_{\Theta_u}$, supported by the positive orthant, is called the spectral measure of $c^{-1}\odot (u\odot X_{0})_+^\alpha$, see \cite{dehaan:resnick:1977} for more details. By identification of the two expressions of the same limit, we obtain the following proposition.
\begin{prop}\label{prop:spec}
With $Y_0$ defined in Proposition \ref{prop:srv}, the distribution of $(u\odot Y_{0})_+^\alpha/\|(u\odot Y_{0})_+\|_\alpha$, if non-degenerate, is the spectral measure of  $c^{-1}\odot ( u\odot X_{0})_+^\alpha\in [0,\infty)^d$. Moreover, it is independent of $\|(u\odot Y_0)_+\|_\alpha$, and $c_+(u)^{-1}\|(u\odot Y_0)_+\|_\alpha$ is standard Pareto distributed.
\end{prop}
\begin{proof}
That $c_+(u)^{-1}\|(u\odot Y_0)_+\|_\alpha$ is standard Pareto distributed follows from the convergence in \eqref{eq:convu} associated with the regularly varying property, ensuring the homogeneity of the limiting measure. Then, using again the homogeneity in \eqref{eq:convu}, it follows that $(u\odot Y_{0})_+^\alpha/\|(u\odot Y_{0})_+\|_\alpha$ and $c_+(u)^{-1}\|(u\odot Y_0)_+\|_\alpha$ are independent.
\end{proof}

\begin{example}[\textbf{Diagonal BEKK-ARCH}, continued]\label{ex:conj}
We have not been able to establish the existence of $Y_0$ satisfying \eqref{eq:srv2}, except the case of the scalar BEKK-ARCH where the diagonal elements of $A$ are identical. In this case the process is a special case of the Similarity BEKK-ARCH, see Example \ref{ex:similarity}. Even in this case, the characterization of the spectral distribution is not an easy task because of the diagonality of $A$, ruling out Theorem 1.4 of \cite{Buraczewski2009}. In Section \ref{sec:spectral estimates} in the appendix we have included some estimates of the spectral measure of $X_0$ for the bivariate case. The plots suggest that the tails of the process are indeed dependent.  We emphasize that new multivariate renewal theory should be developed in order to prove that the Diagonal-ARCH model is VSRV. We leave such task for future research.$\square$
\end{example}

\section{Vector-scaling regularly varying time series and their extremal behavior}\label{sec:srvtimeseries}
The existence of the tail vector in Proposition \ref{prop:srv} allows us to extend the asymptotic results of  \cite{perfekt1997} to VSRV vectors taking possibly negative values. In order to do so, we use the notion of tail chain from \cite{Basrak2009RVMultivariate}  adapted to VSRV stationary sequences with eventually different tail indices.

\subsection{Vector scaling regularly varying time series} \label{subsec:vecrv}
We introduce a new notion of multivariate regularly varying time series based on VSRV of $X_t$. \begin{dfn}\label{def:srv1}
The stationary process $(X_t)$ is  VSRV if and only if  there exists a process $(Y_t)_{t\ge0}$, with  non-degenerate marginals for $Y_0$, such that
\begin{align*}
\mathcal L(((c_it)^{-1/\alpha_i})_{1\le i\le d}\odot (X_0,X_1,\ldots,X_k)\mid \|X_0\|_\alpha>t)\to_{t\to \infty} \mathcal L(Y_0,\ldots,Y_k),
\end{align*}
for all $k\ge 0$. The sequence $(Y_t)_{t\ge 0}$ is called the \emph{tail process}.
\end{dfn}
Following \cite{Basrak2009RVMultivariate}, we extend the notion of spectral measure to the one of \emph{spectral processes}  for any VSRV  stationary process:
\begin{dfn}\label{def:srv2}
The VSRV stationary process $(X_t)$ admits the spectral process $(\Theta_t)$ if and only if
\begin{align*}
\mathcal L(\|X_0\|_\alpha^{-1} (X_0,X_1,\ldots,X_k)\mid \|X_0\|_\alpha>t)\to_{t\to \infty} \mathcal L(\Theta_0,\ldots,\Theta_k),
\end{align*}
for all $k\ge 0$.
\end{dfn}
By arguments similar to the ones in the proof of Proposition \ref{prop:spec}, it follows that the VSRV properties also characterize the spectral process of $(c^{-1}\odot (  u\odot X_{t})_+^\alpha)_{t\ge 0}$, with $X_0$ following the stationary distribution, which has the distribution of $((u\odot Y_{t})_+^\alpha/\|(u\odot Y_{0})_+\|_\alpha)_{t\ge 0}$. We have the following proposition.
\begin{prop}
For a VSRV stationary process $(X_t)$, where $Y_0$ has non-degenerate marginals and $\|Y_0\|_\alpha$ is standard Pareto distributed, the spectral process of any non-degenerate  $(c^{-1}\odot ( u\odot X_{t})_+^\alpha)_{t\ge 0}$  is distributed as $((u\odot Y_{t})_+^\alpha/\|(u\odot Y_{0})_+\|_\alpha)_{t\ge 0}$ and independent of $\|(u\odot Y_0)_+\|_\alpha$. Moreover $c_+(u)^{-1}\|(u\odot Y_0)_+\|_\alpha$ is standard Pareto distributed.
\end{prop}

\subsection{The tail chain} \label{sec:tailchain}

In the following, we will focus on the dynamics of the tail process $(Y_t)_{t\ge 1}$ in Definition \ref{def:srv1}, given the existence of $Y_0$. We will restrict ourselves to the case where $(X_t)$ is a Markov chain, which implies that $(Y_t)$ is also a Markov chain called the \emph{tail chain}; see  \cite{perfekt1997}.
We have the following proposition.

\begin{prop}
Let $(X_t)$ satisfy \eqref{eq:BEKK1}-\eqref{eq:BEKK2} be a VSRV stationary process. With $\tilde{M}_t$ defined in \eqref{eq:BEKK4}, the tail process $(Y_t)$ admits the multiplicative form
\begin{equation}\label{tailchain}
Y_{t+1}=\tilde M_{t+1}Y_t,\qquad t\ge 0.
\end{equation}
\end{prop}
\begin{proof}
Following the approach of \cite{janssen:segers:2014}, one first notices that the existence of the kernel of the tail chain does not depend on the marginal distribution. Thus the characterization of the kernel extends automatically from the usual multivariate regular variation setting to the vector scaling regular variation one. It is straightforward to check Condition 2.2  of \cite{janssen:segers:2014}. We conclude that the tail chain has the multiplicative structure in \eqref{tailchain}.
\end{proof}
The tail chain for VSRV process satisfying \eqref{eq:BEKK1}-\eqref{eq:BEKK2} is the same  no matter the values of the marginal tail indices; for the multivariate regularly varying case with common tail indices it coincides with the tail chain of \cite{janssen:segers:2014} under Condition {\bf M}. Notice that we can extend the tail chain $Y_t$ backward in time ($t< 0$) using Corollary 5.1 of \cite{janssen:segers:2014}.

\subsection{Asymptotic behavior of the maxima}
From the previous section, we have that the tail chain $(Y_t)$ quantifies the extremal behavior of $(X_t)$ in \eqref{eq:BEKK1}-\eqref{eq:BEKK2}. Let us consider the asymptotic behavior of the component-wise maxima
\[
\max(X_1,\ldots,X_{n})=\left(\max(X_{1,k},\ldots,X_{n,k})\right)_{1\le k\le d}.
\]
Let $u=(1,\ldots,1)={\bf1}\in \R^d$ and assume that  $c_+({\bf 1})=\lim_{t\to\infty}\P(X_0 \nleqslant x (t)\mid |X_{0}|\nleqslant x (t))$ is positive.
Recall that for  $(X_t)$ $i.i.d.$, the suitably scaled maxima converge to the Fr\' echet distribution; see \cite{dehaan:resnick:1977}, i.e.
 for any $x=\left(x_1,\ldots,x_d\right)^\intercal\in \R_+^d$, defining $u_n(x)$ such that $n\P(X_{0,i}>u_{n,i}(x))\sim x_i^{-1}$, $1\le i \le d$, we have
\[
\P(\max(X_1,\ldots,X_{n})\le u_n(x))\to \exp(-A^\ast(x)),
\]
if and only if $(X_0)_+$ is vector scaling regularly varying.
In such case, due to Condition {\bf M}, we have the expression
\begin{eqnarray}
A^\ast(x)=c_+({\bf 1})\E\left[\frac1{ \|(Y_0)_+\|_\alpha}\max_{1\le i \le d} \frac{ (Y _{0,i})^{\alpha_k}_+}{c_i x_i} \right] . \label{eq:A_star}
\end{eqnarray}
Let us assume the following Condition, slightly stronger than \eqref{eq:lyapunov_cond}:
\begin{equation}\label{eq:condmom}
\text{There exists $p>0$  such that}\lim_{n\to \infty} \E[\|\tilde M_1\cdots \tilde M_n\|^p]^{1/n}<1.
\end{equation}

\begin{thm}\label{th:extr}
Let $X_t$ satisfy \eqref{eq:BEKK1}-\eqref{eq:BEKK2}. With $\tilde{M}_t$ defined in \eqref{eq:BEKK4}, suppose that condition \eqref{eq:condmom} holds.
 Suppose that the stationary distribution is VSRV. Assuming the existence of $Y_0$ in Definition \ref{def:srv1}, we have that
\[
\P(\max(X_m,\ldots,X_{n})\le u_n(x))\to \exp(-A(x)),
\]
where $A(x)$ admits the expression
\begin{eqnarray}
c_+({\bf 1})\E\left[ \max_{1\le i \le d}\frac{  \max_{k\ge 0} \left(\left( \prod_{1\le j\le k}\tilde M_{k-j} Y_0\right)_{i} \right)^{\alpha_k}_+}{ \|(Y_0)_+\|_\alpha  c_ix_i}  - \max_{1\le i \le d}  \frac{  \max_{k\ge 1}\left(\left( \prod_{1\le j\le k}\tilde M_{k-j} Y_0\right)_i\right)^{\alpha_k}_+}{ \|(Y_0)_+\|_\alpha c_i x_i}   \right] . \label{eq:A}
\end{eqnarray}
\end{thm}
\begin{proof}
We verify the conditions of Theorem 4.5 of \cite{perfekt1997}. Condition B2 of \cite{perfekt1997} is satisfied under the more tractable Condition 2.2 of  \cite{janssen:segers:2014}. Indeed,  the tail chain depends only on the Markov kernel and one can apply Lemma 2.1 of \cite{janssen:segers:2014}, because it extends immediately to the vector scaling regularly varying setting. Condition $\rm D(u_n)$ of \cite{perfekt1997} holds by geometric ergodicity of the Markov chain for a sequence $u_{n}=C\log n$, with $C>0$ sufficiently large. Lastly, the finite clustering condition,
\begin{equation}\label{eq:fc}
\lim_{m\to \infty}{\limsup}_{n\to\infty}\P[\max(|X_m|,\ldots,|X_{C\log n}|)\nleqslant  u_n(x) \mid |X_0|\nleqslant u_n(x)]=0,
\end{equation}
holds for any $C>0$   using the same reasoning as in the proof of Theorem 4.6 of  \cite{Mikosch:Wintenberger2013}  under the drift condition (DC$_p$) for some $p<\alpha=\min\{\alpha_i:1\le i\le d\}$. As $(X_t)$ is also standard $\alpha$ regularly varying, actually the drift condition holds thanks to Condition \eqref{eq:condmom} on some sufficiently large iterations of the Markov kernel. Finally, as
\eqref{eq:fc} is a special case of Condition $\rm D^{\infty}(c\log n)$ of \cite{perfekt1997}, we obtain the desired result with the characterization given in Theorem 4.5 of \cite{perfekt1997}
\begin{align*}
A(x)=\int_{(0,\infty)^{d}\setminus (0,x) }\P\left( T_j\le x,k\ge 1\mid T_0=y\right)\nu(dy),
\end{align*}
where $(T_k)_{k\ge 0}$ is the tail chain of the standardized Markov chain $(c_i^{-1}(X_{k,i})^{\alpha_i}_+)_{1\le i\le d}$, $k\ge 0$.
As $\mu$  restricted to $(0,\infty)^{d}\setminus (0,1)^d$ is the  distribution of $Y_{0}$,  we assume that $x_i\ge 1$ for all $1\le i\le d$ so that we identify $\nu$ as the distribution of
\[
 (c_i^{-1} (Y_{0,i})^{\alpha_i}_+)_{1\le i\le d}\qquad \mbox{under the constraint}\qquad \max_{1\le i\le d}c_i^{-1}( Y_{0,i})^{\alpha_i}_+ /x_i> 1.
\]
Thus we have
\begin{align*}
A(x)&= \P\left( c_i^{-1}( Y_{k,i})^{\alpha_i}_+/x_i\le 1,k\ge 1, 1\le i\le d, \max_{1\le i\le d}c_i^{-1}( Y_{0,i})^{\alpha_i}_+ /x_i> 1\right).
\end{align*}
To  obtain an expression that is valid for any $x_i>0$, we  exploit the homogeneity property, and we obtain
\begin{align*}
A(x)
& =\P\left( \max_{k\ge 0}\max_{1\le i\le d}(c_i x_i)^{-1} Y_{k,i} ^{\alpha_i} > 1\right)- \P\left( \max_{k\ge 1}\max_{1\le i\le d}(c_i x_i)^{-1} Y_{k,i}^{\alpha_i}  >  1\right)\\
& = c_+({\bf 1})\E\left[ \frac{\max_{k\ge 0}\max_{1\le i\le d}(c_i x_i)^{-1} ( Y_{k,i})_+ ^{\alpha_i}  }{\|(Y_{0})_+\|_\alpha }- \frac{\max_{k\ge 1}\max_{1\le i\le d}(c_i x_i)^{-1}  (Y_{k,i})_+ ^{\alpha_i}  }{\|(Y_{0})_+\|_\alpha }  \right]
\end{align*}
because  $c_+({\bf 1})^{-1}\|(Y_{0 })_+\|_\alpha$ is standard Pareto distributed and independent of the spectral process $(Y_k)_+^\alpha/\|(Y_{0 })_+\|_\alpha$. This expression is homogeneous and extends to any possible $x$ by homogeneity.  \end{proof}
\subsection{Extremal indices}
As the random coefficients $\tilde M_t$ in \eqref{eq:BEKK4}  may be large, consecutive values of $X_t$ can be large. In the univariate case, one  says that the extremal values appear in clusters. An indicator of the average length of the cluster is the inverse of the extremal index, an indicator of extremal dependence; see \cite{leadbetter:lindgren:rootzen:1983}. \\
Thus, the natural extension of the extremal index is the function
$\theta(x)=A(x)/A^\ast(x)$, with $A^\ast(x)$ and $A(x)$ defined in \eqref{eq:A_star} and \eqref{eq:A}, respectively.
Notice that there is no reason why $\theta$ should not depend on $x$. When $x_i\ge c_+({\bf 1})$, for $1\le i\le d$, we have the more explicit expression in terms of the spectral process,
\begin{align}\label{eq:extremal1}
\theta(x) = \P\left(  Y_{k,i} ^{\alpha_i}  \le c_i x_i,k\ge 1, 1\le i\le d\mid   Y_{0,i} ^{\alpha_i}   > c_i x_i, 1\le i\le d \right).
\end{align}
However, the extremal index $\theta_i$ of the marginal index $(X_{t,i})$ is still well-defined. It depends on the complete dependence structure of the multivariate Markov chain thanks to the following proposition:
\begin{prop}\label{prop:ei}
Let $X_t$ satisfy \eqref{eq:BEKK1}-\eqref{eq:BEKK2}. With $\tilde{M}_t$ defined in \eqref{eq:BEKK4} satisfying \eqref{eq:condmom} and assuming the existence of $Y_0$ in Definition \ref{def:srv1},  the extremal index, $\theta$, defined in \eqref{eq:extremal1}, is a positive continuous function bounded from above by $1$ that can be extended to $(0,\infty]^d\setminus\{\infty,\ldots,\infty\}$. The extremal indices of the marginals are
\begin{align*}
\theta_i&=\theta(\infty,\ldots,\infty,x_i,\infty,\ldots,\infty)\\
&= \frac{\E\left[   \|(Y_{0})_+\|_\alpha^{-1} \left(\max_{k\ge 0}  \left(\left( \prod_{1\le j\le k}\tilde M_{k-j} Y_0\right)_{i}\right)^{\alpha_i}_+   - \max_{k\ge 1}  \left(\left( \prod_{1\le j\le k}\tilde M_{k-j} Y_0\right)_i\right)^{\alpha_i}_+ \right)\right]}{\E\left[  \|(Y_{0})_+\|_\alpha^{-1} \left(Y_{0,i}\right)^{\alpha_i}_+    \right]}.
\end{align*}
\end{prop}
\begin{proof}
Except for the positivity of the extremal index, the result follows by Proposition 2.5 in \cite{perfekt1997}.
The positivity is ensured by applying Corollary 2 in \cite{segers:2005}.
\end{proof}
\begin{example}[\textbf{Diagonal BEKK-ARCH}, continued]
Suppose that $X_0$ is VSRV as conjectured in Example \ref{ex:conj}. It follows from the tail chain approach of \cite{janssen:segers:2014} that the stationary Markov chain $(X_t)$ is regularly varying. Thanks to the diagonal structure of the matrices $\tilde M_k=A m_k$, one can factorize $\|(Y_0)_+\|_\alpha^{-1}(Y_{0,i})^{\alpha_i}$ in the expression of $\theta_i$ provided in Proposition \ref{prop:ei}. Since $\|(Y_0)_+\|_\alpha^{-1}(Y_{0,i})^{\alpha_i}$ and  $m_k$ are independent for $k\ge1$,  we recover a similar expression as in the remarks after  Theorem 2.1 in \cite{dehaan:1989}:
\begin{align*}
\theta_i = \E\left[ \max_{k\ge 0}  \left(A_{ii}^k\prod_{1\le j\le k} m_{j}\right)^{\alpha_i}_+   - \max_{k\ge 1}  \left(A_{ii}^k \prod_{1\le j\le k} m_{j} \right)^{\alpha_i}_+  \right].
\end{align*}
We did not manage to provide a link between the $\theta_i$ and the extremal index $\theta(x)$ of the (multivariate) stationary solution $(X_t)$ of the Diagonal BEKK-ARCH. Due to the different normalising sequences in the asymptotic extremal result given in  Theorem \ref{th:extr}, the extremal index $\theta(x)$ depends on the constants $c_i,i=1,...,d$. For $x_i^\ast=c_+({\bf 1}) $, $1\le i\le d$, the expression \eqref{eq:extremal1} gets more simple because $c_+({\bf 1})^{-1}\|(Y_0)_+\|_\alpha$ is standard Pareto distributed and supported on $[1,\infty)$:
\[
\theta(x^\ast)= \P\left( A_{ii}^k \prod_{1\le j\le k} m_{j} Y_{0,i} \le (c_i c_+({\bf 1}))^{1/\alpha_i} ,k\ge 1, 1\le i\le d  \right).
\]
One can check  that $\theta(x^\ast)\ge \theta_{i_0}$ where $1\le i_0\le d$ satisfies  $A_{i_0i_0}\ge  A_{ii}$, $1\le i\le d$ so that $i_0$ is the marginal with smallest tail and extremal indices. Thus the inverse of the extremal index of the multidimensional Diagonal BEKK-ARCH is not larger than the largest average length of the marginals clusters. It can be interpreted as the fact that the largest clusters are concentrated along the $i_0$ axis, following the interpretation of the multivariate extremal index given on p$.$ 423 of \cite{beirlant:2006}. $\square$
 \end{example}

\subsection{Convergence of point processes}

Let us consider the vector scaling point process on $\R^d$
\begin{align}\label{eq:Nn}
  N_n(\cdot) = \sum_{t=1}^n \delta_{((c_in)^{-1/\alpha_i})_{1\le i\le d} \odot X_t}(\cdot),\qquad n\ge 0.
\end{align}

We want to characterize the asymptotic distribution of the point process $N_n$ when $n\to \infty$. We refer to  \cite{resnick:2007} for details on the convergence in distribution for random measures. In order to characterize the limit, we adapt the approach of  \cite{davis:1995} to the multivariate VSRV case similar to \cite{davis:1998}. The limit distribution will be a cluster point process admitting the expression
\begin{align}\label{eq:N}
  N(\cdot)=\sum_{j=1}^\infty\sum_{t=1}^\infty\delta_{\big((c_i\Gamma_j)^{-1/\alpha_i}\big)_{1\le i\le d} \odot \; Q_{j,t}}(\cdot),
\end{align}
where $\Gamma_j$, $j=1,2,...$, are arrival times of a standard Poisson process, and $(Q_{j,t})_{t\in \Z}$, $j=1,2,...$, are mutually independent cluster processes. Following \cite{basrak:2016}, we use the back and forth tail chain $(Y_t)$ to describe the cluster process: Consider the process $(Z_t)$, satisfying
\[
\mathcal L\Big((Z_t)_{t\in\Z}\Big)= \mathcal L\Big((Y_t)_{t\in\Z} \mid \sup_{t\le - 1} \|Y_{t}\|_\alpha\le 1\Big),
\]
which is well defined when the anti-clustering condition \eqref{eq:fc} is satisfied. Then we have
\[
\mathcal L\Big((Q_{j,t})_{t\in\Z}\Big)= \mathcal L\Big(L_Z^{-1}(Z_t)_{t\in\Z}    \Big),\qquad j\ge1,
\]
with $L_Z = \sup_{t\in \Z}\|Z_t\|_\alpha$. Notice that the use of the pseudo-norm $\|\cdot\|_\alpha$ and the fact that $\|Y_0\|_\alpha$ is standard Pareto are crucial to mimic the arguments of \cite{basrak:2016}. The limiting distribution of the point process $N_n$ coincides with the one of $N$:
\begin{thm}\label{thm:pointprocess}
 Let $X_t$ satisfy \eqref{eq:BEKK1}-\eqref{eq:BEKK2}. With $\tilde{M}_t$ defined in \eqref{eq:BEKK4}, suppose that \eqref{eq:condmom} holds, and assume that $Y_0$ in Definition \ref{def:srv1} exists. With $N_n$ defined in \eqref{eq:Nn} and $N$ defined in  \eqref{eq:N},
\[
N_n\stackrel{d}{\to} N,\qquad n\to \infty.
\]
\end{thm}
\begin{proof}
Let us denote $sign$ the operator $sign(x)=x/|x|$, $x\in \R$, applied coordinatewise to vectors in $\R^d$.
We apply Theorem 2.8 of \cite{davis:1998} to the transformed process $(c^{-1}\odot sign(X_t)\odot  |X_{t}|^{\alpha})_{t\in \Z}$ which is standard regularly varying of order 1. In order to do so, one has to check that the anti-clustering condition \eqref{eq:fc} is satisfied and that the cluster index of its max-norm is positive. This follows from arguments developed in the proof of Theorem \ref{th:extr}. The mixing condition of \cite{davis:1998} is implied by the geometric ergodicity of $(X_t)$. Thus,  the limiting distribution of the point process $\sum_{t=1}^n \delta_{n^{-1}c^{-1}\odot sign(X_t)\odot  |X_{t}|^{\alpha}}$ coincides with the one of the cluster point process $\sum_{j=1}^\infty\sum_{t=1}^\infty\delta_{\Gamma_j^{-1}\tilde Q_{j,t}}$ for some cluster process $(\tilde Q_{j,t})_{t\in\Z}$. A continuous mapping argument yields the convergence of $N_n$ to $\sum_{j=1}^\infty\sum_{t=1}^\infty\delta_{((c_i\Gamma_j)^{-1/\alpha_i})_{1\le i\le d} \odot \; sign( \tilde Q_{j,t})\odot  |\tilde Q_{j,t}|^\alpha}$. The limiting cluster process coincide with $Q_{j,t}$ in distribution thanks to the definition of VSRV processes.
\end{proof}

\section{Sample covariances} \label{sec:autocov}
In this section, we derive the limiting distribution of the sample covariances for certain BEKK-ARCH processes. Consider the sample covariance matrix,
\begin{eqnarray*}
 \Gamma_{n,X}  = \frac{1}{n}\sum_{t=1}^{n}X_t X_t^\intercal.
\end{eqnarray*}
Let $\kvech(\cdot)$ denote the half-vectorization operator, i.e$.$ for a $d \times d$ matrix $A = [a_{ij}]$,  $\kvech(A)= (a_{11},a_{21},...,a_{d1},a_{22},...,a_{d2},a_{33},...,a_{dd})^\intercal$ $(d(d+1)/2\times 1)$. The derivation of the limiting distribution of the sample covariance matrix relies on using the multidimensional regularly varying properties of the stationary process $(\kvech(X_t X_t^\intercal):t\in \Z)$. Let ${\bf a}_n^{-1}$ denote the normalization matrix,
\[
{\bf a}_n^{-1}=\big(n^{-1/\alpha_i-1/\alpha_j}c_i^{-1/\alpha_i}c_j^{-1/\alpha_j}\big)_{1\le i,j\le d}.
\]
Using Theorem \ref{thm:pointprocess} and adapting the continuous mapping argument of Proposition 3.1 of \cite{davis:1998} yield the following result.
\begin{prop}
 Let $X_t$ satisfy \eqref{eq:BEKK1}-\eqref{eq:BEKK2}. With $\tilde{M}_t$ defined in \eqref{eq:BEKK4} satisfying \eqref{eq:condmom} and assuming the existence of $Y_0$ in Definition \ref{def:srv1}, we have
\[
\sum_{t=1}^n\delta_{\kvech({\bf a}_n^{-1})\odot \kvech(X_t X_t^\intercal)}\stackrel{d}{\to} \sum_{\ell=1}^\infty \sum_{t=1}^\infty \delta_{\kvech ({\bf P}_\ell)  \odot \; \kvech(Q_{\ell,t}Q_{\ell,t}^\intercal)},\qquad n\to \infty,
\]
where
\[
{\bf P}_\ell=\big(\Gamma_\ell^{-1/\alpha_i-1/\alpha_j}c_i^{-1/\alpha_i}c_j^{-1/\alpha_j}\big)_{1\le i,j\le d}.
\]
\end{prop}
Let us define $\alpha_{i,j}=\alpha_i\alpha_j/(\alpha_i+\alpha_j)$ and assume that $\alpha_{i,j}\neq 1$ and $\alpha_{i,j}\neq 2$ for all $1\le i\le j\le d$. Note that $\alpha_{i,j}$ is a candidate for the tail index of the cross product $X_{t,i}X_{t,j}$ and that $\alpha_{i,i}=\alpha_i/2$, $1\le i\le d$. Actually it is the case under some extra assumptions ensuring that the product $Y_{0,i}Y_{0,j}$ is non null, see Proposition 7.6 of \cite{resnick:2007}. In line with Theorem 3.5 of \cite{davis:1998}, we then get our main result on the asymptotic behavior of the empirical covariance matrix
\begin{thm}\label{th:empcov}
 Let $X_t$ satisfy \eqref{eq:BEKK1}-\eqref{eq:BEKK2}. With $\tilde{M}_t$ defined in \eqref{eq:BEKK4}, suppose that \eqref{eq:condmom} holds, and assume that $Y_0$ in Definition \ref{def:srv1} exists. Moreover, for any $(i,j)$ such that $1<\alpha_{i,j}<2$, suppose that
\begin{equation}\label{cond:var}
\lim_{\varepsilon\to 0}\lim\sup_{n\to \infty}\var\big(n^{-1/\alpha_{i,j}}\sum_{t=1}^nX_{t,i}X_{t,j}\mathbf{1}_{|X_{t,i}X_{t,j}|\le n^{1/\alpha_{i,j}}\varepsilon}\big)=0.
\end{equation}
Then
\[
\Big( \sqrt n \wedge n^{1-1/\alpha_{i,j}}  (\Gamma_{n,X}-\E[\Gamma_{n,X}]\mathbf{1}_{\alpha_{i,j}>1})_{i,j}\Big)_{1\le j \le i\le d} \stackrel{d}{\to} S,\qquad n\to \infty,
\]
where $S_{i,j}$ is an $\alpha_{i,j}\wedge 2$-stable random variable for  $1\le i\le j\le d$ and non-degenerate for $i=j$.
\end{thm}

When Theorem \ref{th:empcov} applies, as $\alpha_{i,j}\ge (\alpha_i \wedge \alpha_j)/2$, the widest confidence interval on the covariance estimates is supported by the $i_0$th marginal satisfying $\alpha_{i_0}\le \alpha_i$ for all $1\le i\le d$.

In order to apply Theorem \ref{th:empcov}, the main difficulty is to show that the condition \eqref{cond:var} holds. However, notice that  Theorem \ref{th:empcov} applies simultaneously on the cross-products with $\alpha_{i,j}\notin [1,2]$ with no extra assumption. Next, we apply Theorem \ref{th:empcov} to the ongoing examples.

\begin{example}[\textbf{Diagonal BEKK-ARCH}, continued] \label{ex:covdiag}
Consider the diagonal BEKK-ARCH process and the cross products $X_{t,i}X_{t,j}$ for some $i\le j$ and any $t\in \Z$.  From Hölder's inequality (which turns out to be an equality in our case), we have
\[
\E[|A_{ii}A_{jj}m_0^2|^{\alpha_{i,j}}]=\E[|A_{ii}m_0|^{\alpha_i}|]^{\alpha_{i,j}/\alpha_{i}}\E[|A_{jj}m_0|^{\alpha_{j}}]^{\alpha_{i,j}/\alpha_{j}}=1.
\]
Thus,  $(X_{t,i}X_{t,j})$, which is a function of the Markov chain $(X_t)$, satisfies the drift condition (DC$_p$) of \cite{Mikosch:Wintenberger2013} for all $p<\alpha_{i,j}$. Then, one can  show that \eqref{cond:var} is satisfied using the same reasoning as in the proof of Theorem 4.6 of  \cite{Mikosch:Wintenberger2013}. $\square$
\end{example}

\begin{example}[\textbf{Similarity BEKK-ARCH}, continued]\label{ex:covsimilarity}
  If $\alpha_{i,j}\notin [1,2]$, the limiting distribution of the sample covariance matrix for the Similarity BEKK-ARCH follows directly from Theorem \ref{th:empcov}. If $\alpha_{i,j}\in (1,2)$ the additional condition \eqref{cond:var} has to be checked. Relying on the same arguments as in Example \ref{ex:covdiag}, one would have to verify that the condition (DC$_p$) of \cite{Mikosch:Wintenberger2013} holds for the Similarity BEKK-ARCH process, which appears a difficult task as it requires to find a suitable multivariate Lyapunov function. We leave such task for future investigation. Consider the special case of the scalar BEKK-ARCH process introduced in Example \ref{ex:similarity}. Here $A=\sqrt{a}I_{d}$, with $I_{d}$ the identity matrix, such that $\tilde M_t$ is diagonal. In the case $\alpha_{i,j}\in (1,2)$ for a least some pair $(i,j)$, the limiting distribution of the sample covariance is derived along the lines of Example \ref{ex:covdiag}. Specifically, this relies on assuming that $a<\exp\left\{ (1/2)\left[-\psi(1)+\log(2)\right]\right\}$ such that a stationary solution exists, and noting that the index of regular variation for each marginal of $X_t$ is given by $\alpha$ satisfying $\E[|\sqrt{a}m_t|^{\alpha}]=1$. $\square$
\end{example}

\begin{example}[\textbf{ID BEKK-ARCH}, continued]
  Whenever $\alpha_{i,j}\notin [1,2]$, the limiting distribution of the sample covariance matrix for the ID BEKK-ARCH follows directly from Theorem \ref{th:empcov}. Similar to Example \ref{ex:covsimilarity} we leave for future investigation to show whether condition \eqref{cond:var} holds. $\square$
\end{example}

The previous examples are important in relation to variance targeting estimation of the BEKK-ARCH model, as considered in \cite{Pedersen2013VTBEKK}. For the univariate GARCH process, \cite{vaynman2013stableGARCH} have shown that the limiting distribution of the (suitably scaled) variance targeting estimator follows a singular stable distribution when the tail index of the process lies in $(2,4)$. We expect a similar result to hold for the BEKK-ARCH process.
%
%
%
%
%
%
\section{Concluding remarks} \label{sec:conclusion}

We have found a mild sufficient condition for geometric ergodicity of a class of BEKK-ARCH processes. By exploiting the the processes can be written as a multivaraite stochastic recurrence equation (SRE), we have investigated the tail behavior of the invariant distribution for different BEKK-ARCH processes. Specifically, we have demonstrated that existing Kesten-type results apply in certain cases, implying that each marginal of the invariant distribution has the same tail index. Moreover, we have shown for certain empirically relevant processes, existing renewal theory is not applicable. In particular, we show that the Diagonal BEKK-ARCH processes may have component-wise different tail indices. In light of this property, we introduce the notion of vector scaling regular varying (VSRV) distributions and processes. We study the extremal behavior of such processes and provide results for convergence of point processes based on VSRV processes. It is conjectured, and supported by simulations, that the Diagonal BEKK-ARCH process is VSRV. However, it remains an open task to verify formally that the property holds. Such task will require the development of new multivariate renewal theory.

Our results are expected to be important for future research related to the statistical analysis of the Diagonal BEKK-ARCH model. As recently shown by \cite{Avarucci2012ET}, the (suitably scaled) maximum likelihood estimator for the general BEKK-ARCH model (with $l=1$) does only have a Gaussian limiting distribution, if  the second-order moments of $X_t$ is finite. In order to obtain the limiting distribution in the presence of very heavy tails, i.e. when $\E[\Vert X_t \Vert^2]=\infty$, we believe that non-standard arguments are needed, and in particular the knowledge of the tail-behavior is expected to be crucial for the analysis. We leave additional considerations in this direction to future research.

\bibliographystyle{ecta}
\bibliography{BEKKARCHbib-3}

\newpage{}

\appendix

\section{Appendix}

\subsection{Theorem 1.1 of \cite{alsmeyer2012sre}} \label{sec:Alsmeyer}

Consider the general SRE
\begin{equation}\label{eq:SRE_appendix}
  Y_t = A_tY_{t-1} + B_t
\end{equation}
with $(A_t,B_t)$ a sequence of i.i.d. random variables with generic copy $(A,B)$ such that $A$ is a $d \times d$ real matrix and $B$ takes values in $\mathbb{R}^d$. Consider the following conditions of  \cite{alsmeyer2012sre}:
\begin{itemize}
\item {\bf (A1)} $\E[\log^+(\Vert A \Vert)]<\infty$, where $\Vert \cdot \Vert$ denotes the operator norm.
\item {\bf (A2)} $\E[\log^+(\Vert B \Vert)]<\infty$.
\item {\bf (A3)} $\P[A \in GL(d,\mathbb{R})] = 1$.
\item {\bf (A4)} $\max_{n\in\mathbb{N}}\mathbb{P}\left\{ \Vert x^\intercal\prod_{i=1}^{n} A_{i}\Vert^{-1}\left(x^\intercal\prod_{i=1}^{n}A_{i}\right)\in U\right\} >0,$ for any $x\in\mathbb{S}^{d-1}$ and any non-empty open subset $U$ of $\mathbb{S}^{d-1}$.
\item {\bf (A5)} Let $\mathcal{V}_{\delta}$ denote the open $\delta$-ball in $GL(d,\mathbb{R})$ and let $\mathbb{LEB}$ denote the Lebesgue measure on $M(d,\mathbb{R})$. It holds that for any Borel set $A \in M(d,\mathbb{R})$, $\P(\prod_{i=1}^{n_0} A_{i} \in A) \geq \gamma_0 1_{\mathcal{V}_c(\Gamma_0)}(A)\mathbb{LEB}(A)$ for some $\Gamma_0 \in GL(d,\mathbb{R})$, $n_0\in \mathbb{N}$, and $c,\gamma_0>0$.
\item {\bf (A6)} $\P(A_0v+B_0=v)<1$ for any $v\in \mathbb{R}^d$.
\item {\bf (A7)} There exists $\kappa_0>0$ such that
\begin{equation*}
  \E[\inf_{x\in \mathbb{S}^{d-1}}\Vert x^{\intercal} A_0  \Vert^{\kappa_0}]\geq 1, \quad \E[\Vert A_0 \Vert^{\kappa_0}\log^+\Vert A_0 \Vert ]<\infty, \quad \text{and}\quad 0<\E[\Vert B_0 \Vert^{\kappa_0}]<\infty.
\end{equation*}
\end{itemize}

\begin{thm}[{\citet[Theorem 1.1]{alsmeyer2012sre}}] \label{thm_alsmeyermentemeier}
  Consider the SRE in \eqref{eq:SRE_appendix}) suppose that $\beta := \lim_{n\rightarrow \infty}n^{-1}\log (\Vert \prod_{i=1}^{n} A_{i} \Vert )<0 $ and that {\bf (A1)}-{\bf (A7)} hold, then there exists a unique $\kappa \in (0,\kappa_0]$ such that
\begin{equation*}
  \lim_{n\rightarrow \infty}n^{-1}\log (\Vert \prod_{i=1}^{n} A_{i} \Vert^{\kappa} )=0.
\end{equation*}
Moreover, the SRE has a strictly stationary solution satisfying,
\begin{equation*}
  \lim_{t\rightarrow\infty}t^\kappa\P(x^{\intercal}Y_0>t)=K(x)\quad \text{for all $x\in \mathbb{S}^{d-1}$},
\end{equation*}
where $K$ is a finite positive and continuous function on $\mathbb{S}^{d-1}$.
\end{thm}

\newpage
\subsection{Estimation of the spectral measure for the bivariate diagonal BEKK-ARCH process} \label{sec:spectral estimates}
In this section we consider the estimation of the spectral measure of the diagonal BEKK-ARCH process presented in Example \ref{ex:diagonal2}. Specifically, we consider a special case of the BEKK-ARCH process in \eqref{eq:BEKK1}-\eqref{eq:BEKK2}, where $d=2$:
\[
X_{t}=m_{t}AX_{t-1}+Q_{t},
\]
with $\{Q_{t}:t\in \mathbb{N}\}$ an $i.i.d.$ process with $Q_{t}\sim N(0,C)$ independent of $\{m_{t}:t\in \mathbb{N}\}$, and
\[
A=\begin{pmatrix}A_{11} & 0\\
0 & A_{22}
\end{pmatrix}.
\]

Following the approach    for {\it i.i.d.} sequences of vectors given in \cite{einmahl2001}, we consider the following estimator of the spectral measure of $X_t=(X_{t,1},X_{t,2})^\intercal$:
\[
\hat{\Phi}(\theta)=\frac{1}{k}\sum_{t=1}^{T}\mathbf{1}_{\{R_{t}^{(1)}\lor R_{t}^{(2)}\geq T+1-k,\arctan\frac{T+1-R_{t}^{(2)}}{T+1-R_{t}^{(1)}}\leq\theta\}},\quad\theta\in[0,\pi/2],
\]
where $R_{t}^{(j)}$ denotes the rank of $X_{t,j}$ among $X_{1,j},...,X_{T,j}$, $j=1,2$,
i.e.
\[
R_{t}^{(j)}:=\sum_{i=1}^{T}\mathbf{1}_{\{X_{i,j}\geq X_{t,j}\}}.
\]
Here $k$ is a sequence satisfying $k(T)\rightarrow\infty$ and $k(T)=o(T)$. \cite{einmahl2001} showed that this estimator is consistent for i.i.d. series. We expect a similar result to hold for geometrically ergodic processes. The reason is that the asymptotic behavior of the empirical tail process  used in \cite{einmahl2001}  has been extended to such cases in \cite{kulik:2015}.

We consider the estimation of the spectral measure for different values of $C$, $A_{11}$, and $A_{22}$. In particular, the matrix $C$ is
\[
C=10^{-5}\begin{bmatrix}1 & c\\
c & 1
\end{bmatrix},\quad c\in\{0,0.5\},
\]
and the values $A_{11}$ and $A_{22}$ are determined according to choices of the tail indices of $X_{t,1}$ and $X_{t,2}$, respectively. I.e. $A_{11}$ and $A_{22}$ satisfy $\mathbb{E}[|m_t|^{\alpha_{i}}]=|A_{ii}|^{-\alpha_{i}}$ and are determined by analytical integration. Specifically, with $\phi(\cdot)$ the pdf of the standard normal distribution,
\begin{eqnarray*}
\alpha_{i} & = & 0.5\Rightarrow A_{ii}=(\int_{-\infty}^{\infty}\vert m\vert^{0.5}\phi(m)dm)^{-1/0.5}\approx1.479\\
\alpha_{i} & = & 2.0\Rightarrow A_{ii}=1\\
\alpha_{i} & = & 3.0\Rightarrow A_{ii}=(8/\pi)^{-1/6}\approx0.8557\\
\alpha_{i} & = & 4.0\Rightarrow A_{ii}=3^{-1/4}\approx0.7598
\end{eqnarray*}

Figure \ref{fig1} contains plots of the estimates of the spectral measure. The estimates $\hat{\Phi}(\theta)$ are based on one realization of the process with $T=$ 2,000 and a burn-in period of 10,000 observations.

\begin{figure}[!ht]
  \centering
\includegraphics[width=\textwidth,  clip=TRUE]{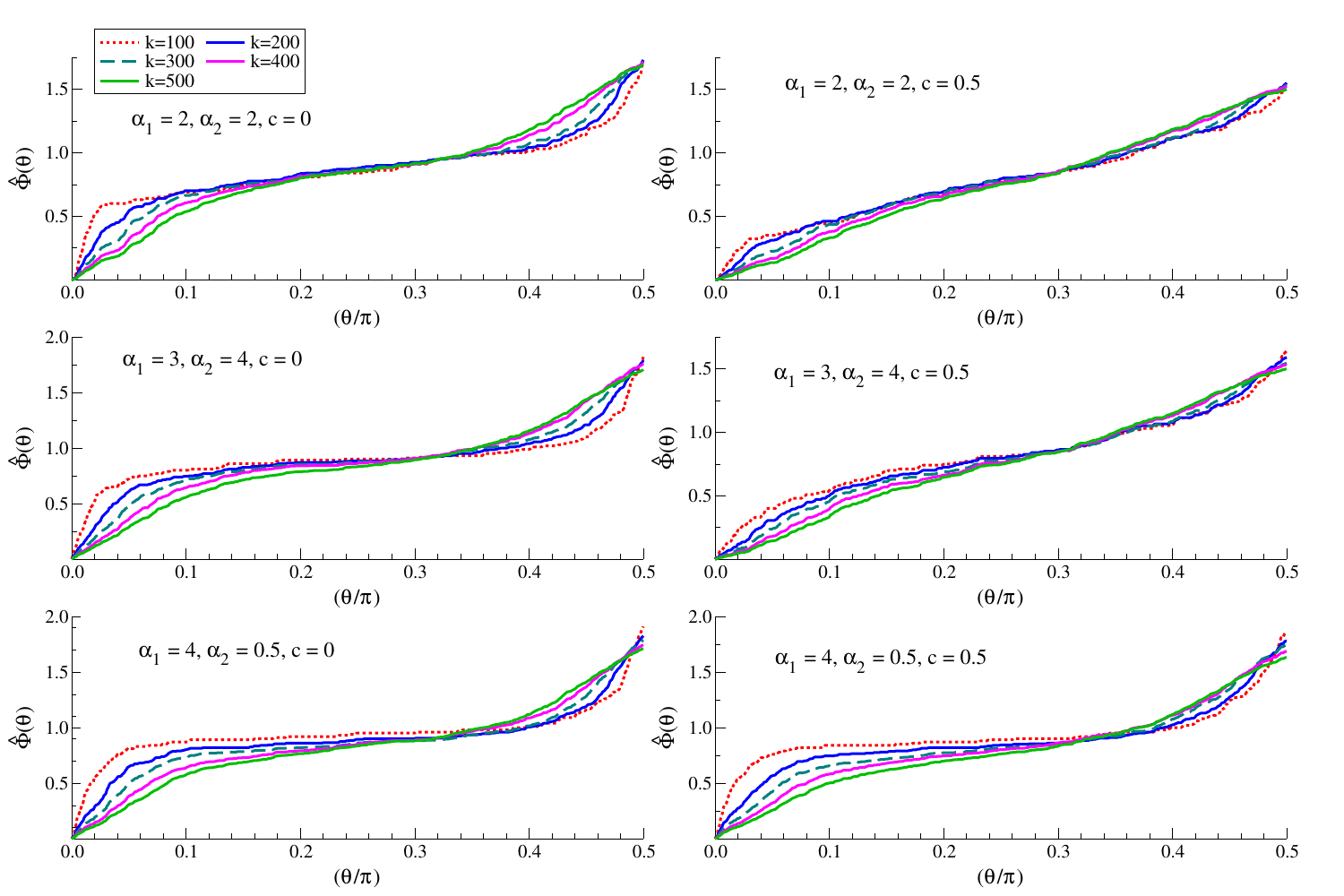}
  \caption{Nonparametric estimates for $k=100,200,300,400,500$ and for various choices of $\alpha_1$, $\alpha_2$, and $c$.}\label{fig1}
\end{figure}

%

\end{document}